\documentclass[12pt,a4paper,reqno]{amsart}
\usepackage{amsmath, amsthm, amscd, amsfonts, amssymb}
\usepackage{comment}
\usepackage{xcolor}

\newtheorem{theorem}{Theorem}[section]

\newtheorem{proposition}[theorem]{Proposition}

\newtheorem*{claim*}{Claim}

\let\R\relax
\let\L\relax

\let\c\relax

\newcommand{\R}{\ensuremath{\mathbb{R}}}
\newcommand{\L}{\ensuremath{\mathbb{L}}}
\newcommand{\cC}{\ensuremath{\mathcal{C}}}
\newcommand{\cD}{\ensuremath{\mathcal{D}}}
\newcommand{\cE}{\ensuremath{\mathcal{E}}}

\newcommand{\cS}{\ensuremath{\mathcal{S}}}
\newcommand{\cT}{\ensuremath{\mathcal{T}}}

\DeclareMathOperator{\spann}{span}

\DeclareMathOperator{\im}{Im}

\DeclareMathOperator{\c}{C}

\title[Semi-Riemannian hypersurfaces in $\L^{n+1}$ with ...]{Semi-Riemannian hypersurfaces in $\L^{n+1}$ with a totally geodesic foliation of codimension one}

\author[S.\ M.\ B.\ Kashani]{S.M.B. Kashani}
\author[M.\ J.\ Vanaei]{M.J. Vanaei}
\author[S.\ M.\ Yaghoubi]{S.M. Yaghoobi}

\address{Dept. of Pure Math., Fculty of Math. Sciences, Tarbiat Modares University, P.O. Box: 14115-134, Tehran, Iran}

\email{kashanim@modares.ac.ir}
\email{javad.vanaei@modares.ac.ir}
\email{seyed.yaghoobi@modares.ac.ir}

\begin{document}



\begin{abstract}
We classify hypersurfaces of the Minkowski space $\L^{n+1}$ that carry a totally geodesic foliation with complete leaves of codimension one. We prove that such a hypersurface is ruled, or a partial tube over a curve or contains a two or three dimensional strip. Moreover, if the hypersurface is embedded then it is a partial tube over a curve. 
\end{abstract}

\maketitle 

\section{Introduction}
The study of submanifolds, specially hypersurfaces in space forms, is one of the main classical problems in differential geometry.
It is known that on hypersurfaces with the same constant curvature as the ambient space form, leaves of the relative nullity distribution form a totally geodesic complete foliation of codimension one on the set of points with minimum relative nullity. Under some hypothesis this foliation can lead to a decomposition of the hypersurface as a product of an immersion of lower dimension with an identity factor. In \cite{HarNir59} codimension one isometric immersions between Euclidean spaces are classified as cylinders over plane curves, namely, all such isometric immersions have the form $id \times g : \R^{n-1} \times \R \to \R^{n-1} \times \R^2$ with orthogonal factors where $g : \R \to \R^2$ is a unit speed plane curve. In the above decomposition leaves of the minimum relative nullity foliation yield the Euclidean factor $\R^{n-1}$. Corresponding results have been given for the Minkowski space $\L^m = \R^m_1$ in \cite{Gra79b}. When the induced metric on the relative nullities of a  codimension one isometric immersion between Minkowski spaces is non-degenerate, the result and the proof are basically similar to the Euclidean case. However, in the case of degenerate relative nullities different techniques have been applied and it turns out that the immersion decomposes as $id \times g : \R^{n-2} \times \L^2 \to \R^{n-2} \times \L^3$ where $g : \L^2 \to \L^3$ is an isometric immersion which can be described using a Cartan frame over a null curve in $\L^3$ and is known as a $B$-scroll.

In \cite{DajRovToj15} a generalized case was addressed as the problem of determining Euclidean hypersurfaces of dimension at least three which admit complete totally geodesic foliations of codimension one. It is shown in \cite{DajRovToj15} that such hypersurfaces in the Euclidean space are either ruled or partial tube over a curve, otherwise they contain a surface-like strip. Some examples of such hypersurfaces are as follows (see \cite{DajRovToj15} for more details): the rulings of a ruled hypersurface form a totally geodesic foliation of codimension one. Flat hypersurfaces given by isometric immersions $f : U \subset \R^n_r \to \R^{n+1}_s$ ($r , s = 0 , 1$) are examples of ruled hypersurfaces which are foliated by (open subsets of) affine subspaces of $\R^{n+1}_s$. 

Let $\gamma : I \to \R^{n+1}_s$ be a smooth curve and $N^{n-1}$ be a hypersurface of the (affine) normal space to $\gamma$ at some point $\gamma(t_0)$ with $t_0 \in I$. A hypersurface can be obtained in $\R^{n+1}_s$ by parallel transporting $N^{n-1}$ along $\gamma$ with respect to the normal connection. Such a hypersurface is called a partial tube over a curve (see Section~\ref{preliminaries} for further details). The parallel displacements of $N$ give rise to a codimension one totally geodesic foliation of $M$, whose leaves are complete whenever $N$ is complete.

An example of a surface-like hypersurface can be given by $M^n = \R^{n-2}_s \times L^2$ with $f = id \times g : M^n \to \R^{n+1}_{s+r}$ where $g : L^2 \to \R^3_r$ is a semi-Riemannian surface. If $\cD_0$ denote the one-dimensional distribution on $L^2$ derived from a non-degenerate foliation of $L^2$ by geodesics, then the codimension one distribution $\cD = \R^{n-2}_s \oplus \cD_0$ defines a totally geodesic foliation on $M^n$ whose leaves are complete whenever the leaves of $\cD_0$ are complete.

In this paper we consider the above problem in the Minkowski space. The main difficulty in classifying such hypesurfaces in non-Riemannian space forms comes from the fact that the relative nullity distribution may be degenerate at some points. To handle the degenerate case we use a totally geodesic screen distribution instead of the realtive nullities. In Proposition~\ref{t.g screen distribution} we make such a screen distribution. Then we generalize the notion of conullity operator to screen distributions and show that in the Minkowski space, apart from the three possibilities in the Euclidiean case, the hypersurface may contain a three dimensional strip which corresponds to the case of degenerate relative nullities. This makes our approach less frame-dependent comparing to \cite{DajRovToj15} in which an appropriate local frame compatible with the codimension one foliation and the minimum relative nullity foliation was constructed and used. Our main theorem states:

\begin{theorem}\label{MainTheorem}
Let $f : M^n_s \to \L^{n+1}$ be a connected nowhere flat isometric immersion from a semi-Riemannian manifold $M^n_s$, $s=0, 1$, carrying a semi-Riemannian totally geodesic foliation of codimension one and index $0$ or $1$. Then $f(M)$  
\begin{itemize}
\item[$(i)$] is a ruled hypersurface, or
\item[$(ii)$] is a partial tube over a curve, or
\item[$(iii)$] contains a $2$-dimensional Riemannian or Lorentzian strip, or
\item[$(iv)$] contains a $3$-dimensional Lorentzian strip.
\end{itemize}
\end{theorem}

An example of hypersurfaces satisfying $(iv)$ are those containing a three-dimensional strip whose non-Euclidean factor is a generalized cylinder over a null curve in $\L^{n+1}$ and the minimal polynomial of its shape operator is $x^2$. Such cylinders are modeled locally as follows: let $\gamma(t)$ be a null curve in $\L^{r+s+3}$ and take a pseudo-orthonormal frame 
\[
\{ \xi(t), \zeta(t), X(t), Y_1(t), \ldots, Y_r(t), Z_1(t), \ldots, Z_s(t) \}
\]
along $\gamma$ such that  
\begin{align*}
& \gamma'(t) = \xi(t) ,\\
& X'(t) = -B(t)\zeta(t), \quad \text{with} \ B \neq 0,\\
& Z'_j(t) \in \spann\{\zeta(t), Z_1(t), \ldots, Z_s(t) \}, \quad 1 \leq j \leq s,
\end{align*}

Then the parametrized hypersurface in $\L^{r+s+3}$ defined locally around the origin by 
\begin{align*}
h(t, u, y_1, \ldots, y_r, z_1, \ldots, z_s) = & \gamma(t) + u\zeta(t) + \sum_iy_iY_i(t) + \sum_jz_jZ_j(t) \\
& \frac{1}{a}X(t) - \sqrt{\frac{1}{a^2} - \sum_iz_i^2}X(t)
\end{align*}
for some $a > 0$, is called a generalized cylinder (of type 1, according to \cite{Mag85}). Its shape operator with respect to the normal vector field 
\[
\eta = a\left( \sqrt{\frac{1}{a^2} - \sum_iz_i^2}\right) X(t) - a\sum_iz_iZ_i(t)
\]
has minimal polynomial $x^2(x-a)$ if $s > 0$ and $x^2$ if $s = 0$. 

Let $\cD_0$ be the two-dimensional distribution of tangent spaces of a non-degenrate totally geodesic foliation of a generalized cylinder $g : L^3 \to \L^4$. Then one can see that the codimension one distribution $\cD = \R^{n-2} \oplus \cD_0$ on the hypersurface $f : M^n = \R^{n-3} \times L^3 \to \L^{n+1}$ with $f = id \times g$ is totally geodesic.

In the local version of the above problem, only one further class of examples may occur. Indeed, it is proved in \cite{DajRovToj15} that in the setting of Theorem~\ref{MainTheorem} in $\R^{n+1}$, $M$ is locally a ruled hypersurface, a partial tube over a curve or an envelope of one-parameter families of flat hypersurfaces, as long as it does not contain any strip. Having Theorem~\ref{MainTheorem}, one can see that a similar approach as in the Euclidean case leads to the same result in the Minkowski space.

Moreover, we prove in Theorem~\ref{embedded} that if $M$ is an embedded hyprsurface then it is a partial tube over a curve. We also show that the Einstein and conformally flat hypersurfaces carrying a semi-Riemannian totally geodesic foliation of codimension one, do not contain any strip (Proposition~\ref{Ein.Conf}). 

The preliminaries are given in the next section and the proofs are given in Section~\ref{ProofMainTheorem}.

Throughout the paper, unless otherwise stated, $M^n_s = M^n$, $s = 0, 1$, denotes an $n$-dimensional smooth manifolds with a semi-Riemannian metric of index $0$ or $1$.
\section{Preliminaries}\label{preliminaries}
\noindent
The content of this section can be found in \cite{AbeMag86, BerConOlm16, CarWes95, DajRovToj15, One83, Toj16}.

Let $\R^m_s$ denote the $m$-dimensional vector space $\R^m$ considered as a smooth $m$-manifold equipped with the metric $g = -\Sigma_1^s dx_i^2 + \Sigma_{s+1}^m dx_i^2$ where $(x_1, \ldots, x_m)$ is the standard coordinate system in $\R^m$. In particular, $\R^m = \R^m_0$ is the Euclidean space and the Lorentzian space form $\L^m = \R^m_1$ is known as the Minkowski space.  

A submanifold $M^n$ in the Euclidean space $\R^m$ or the Minkowski space $\L^m$ is said to be \textit{ruled} if there is a foliation on $M$ by by totally geodesic submanifolds of the ambient space of dimension $n-1$. 

Let $g : L^k \to \R^{k+1}_s$ be an isometric immersion. Then we call the product immersion $id \times g : \R^{n-k}_r \times L^k \to \R^{n+1}_{r+s} = \R^{n-k}_r \times \R^{k+1}_s$ a $k$-dimensional strip (in $\R^{n+1}_{r+s}$). We say that a hypersurface $f : M^n \to \bar{M}^{n+1}$ or, simply $f(M)$, contains a $k$-dimensional strip if an open subset $U \subset M$ is a $k$-dimensional strip, namely, $U$ is isometric to a product $\R^{n-k}_r \times L^k$ and $f|_U$ splits into $f|_U = id \times g$, where $g: L^k \to \R^{k+1}_s$ is an isometric immersion and $id = id_{\R^{n-k}_r}$ is the identity map on $\R^{n-k}_r$. 

\subsection*{Partial Tube}
Partial tubes were first introduced by Carter and West in \cite{CarWes95}. Here we recall the definition of the special case of \textit{partial tube over a curve} which is needed in our results. 

Let $r, s \in \{0,1\}$ and $\gamma : I \to \R^m_s$ be a unit speed curve and assume that there exists an orthonormal frame $\{ \eta_1, \ldots, \eta_k \}$ along $\gamma$. In particular, the vector subbundle $\cE = \spann \{ \eta_1, \ldots, \eta_k \}$ is parallel and flat and the map $\phi : I \times \R^k_r \to \R^m_s$ defined by 
\[
\phi_t(y) = \phi(t, y) = \sum\limits_{i=1}^k y_i \eta_i(t) ,
\]
with $t \in I$ and $y =(y_i)_1^k \in \R^k_r$, is a parallel vector bundle isometry. Let $f_0 : N^{n-1} \to \R^k_r$ be an isometric immersion whose codimension is not reducible. Then the map 
\[
f(p,t) = \gamma (t) + \phi_t(f_0 (p))
\]
from $M^n = N^{n-1} \times I$ to $\R^m_s$ is an immersion whenever $f_0(N) \subset \Omega(\gamma;\phi)$, where
\[
\Omega(\gamma;\phi) = \{y \in \R^m_s : g(\gamma''(t) , \phi_t(y)) \neq 1 \ \ \text{for any} \  t \in I \}.
\]
In this case $f(M)$ is called the \textit{partial tube over} $\gamma$ with fiber $f_0$. With the induced metric on $M^n$, the distribution formed by the tangent spaces to the first factor of $M^n = N^{n-1} \times I$ is totally geodesic in $TM$ and the second fundamental form of $f$ satisfies $\alpha_f(X, \partial /\partial t) = 0$ for any $X \in TN$, where $\partial /\partial t$ is a unit vector field tangent to the factor $I$. The converse is also true as follows.

A set of integrable subbundles $(\cE_i)_{i = 1}^r$ of the tangent bundle $TM$  of a connected manifold $M$ is called a \textit{net} on $M$ if $TM = \bigoplus_{i=0}^r \cE_i$. An \textit{orthogonal net} on $M$ is a net $(\cE_i)_{i = 1}^r$ with mutually orthogonoal subbundles $\cE_i$. On a product $N\times L$, more generally on a twisted product $N \times_\rho L$, a \textit{product net}, is naturally defined whose subbundles are the tangent spaces of the factors. Recall that a twisted product is a product manifold $N \times L$ with a twisted product metric $dx^2 + \rho^2 dy^2$ where $dx^2$ and $dy^2$ are fixed metrics on $N$ and $L$, respectively, and $\rho \in C^\infty(N \times L)$. The following theorem is an special case of Theorem~1 in \cite{PonRec93}:

\begin{theorem}\label{twisted}
If a semi-Riemannian manifold $M$ carries a codimension one totally geodesic foliation, then it is locally isometric to a twisted product $N \times_\rho I$, where $I \subset \R$ is an open interval. The isometry is globally defined if $M$ is simply connected and the leaves of the foliation are complete.
\end{theorem}

The second fundamental form $\alpha$ of $M$ is said to be \textit{adapted} to a given net $(\cE_i)_{i = 1}^r$ on $M$ if $\alpha(\cE_i , \cE_j) = 0$ for all $1 \leq i,j \leq r$ with $i \neq j$.

According to \cite{Toj16} the induced metric on a partial tube over a curve is twisted and its second fundamental form is adapted to the product net. The following theorem shows that the converse is also true. It is stated in \cite{DajRovToj15} as a consequence of the results in \cite{Toj16} for the Riemannian case, but one can check that the result holds in the Lorentzian case, as well.

\begin{theorem}\label{twisted-partialtube}
Let $f : M^n \to \R^m_s$, $s = 0 \ \text{or\ }1$, be an isometric immersion of a twisted product $M^n = N^{n-1} \times_\rho I$, where $N^{n-1}$ is an $(n-1)$dimensional semi-Riemannian manifold and $I \subset \R$ is an open interval and $\rho \in C^\infty(M)$. Assume that the second fundamental form $\alpha_f$ of $f$ is adapted to the product net. Then $f(M)$ is a partial tube over a curve.
\end{theorem}

Let $f:M^n \to \bar{M}^{n+1}(c)$ be an isomteric immersion from a semi-Riemannian manifold $M$ into the space form $\bar{M}^{n+1}(c)$ of constant curvature $c$. We denote by $\alpha_f$ the second fundamental form of $f$ and by $S$ its shape operator on $M$. 

The distribution $\cT$ defined by $\cT_p = \{ x \in T_pM : \alpha_f(x , y) = 0, \forall y \in T_pM \}$ for every $p \in M$, is known as the \textit{relative nullity distribution} on $M$. The relative nullity of $f$ at a $p$ is the dimension $\nu(p)$ of $\cT_p$ and the minimum value of $\nu(p)$ on $M$, denoted by $\nu_0$, is called the \textit{index of relative nullity}.

Some properties of the relative nullity distribution are given in the following theorem.

\begin{theorem}[\cite{AbeMag86}]\label{W-properties}
Let $W$ denote the set of points in $M$ with the minimum relative nullity $\nu_0$. Then
\begin{itemize}
\item[$(i)$]
$W$ is an open subset of $M$;

\item[$(ii)$]
$\cT$ defines a smooth and involutive distribution in $W$;

\item[$(iii)$]
$\cT$ is a totally geodesic distribution in $M$; and

\item[$(iv)$] 
each leaf of $\cT$ is an immersed totally geodesic submanifold of $\bar{M}$.
\end{itemize}
\end{theorem}

Moreover, it is also proved in \cite{AbeMag86} that if $M$ is complete then the leaves of the relative nullity distribution in $W$ are complete, as well.


\section{Main Result} \label{ProofMainTheorem}
In this section we assume that $f : M^n \to \L^{n+1}$ is a connected nowhere flat isometric immersion and $\cD$ is a semi-Riemannian totally geodesic distribution of codimension one on $M$ with complete leaves. For simplicity, we use the same notation $g$ for the metric tensors on $\L^{n+1}$ and $M$ and the induced metrics on their submanifolds. 

We first prove that as in the Euclidean case there are three possibilities at any point of $M$: the second fundamental form $\alpha_f$ is adapted to $(\cD , \cD^\perp)$ or the relative nullity space of $\alpha_f$ intersects $\cD$ in a subspace of codimension $0$ or $1$. Indeed, since $\cD$ is totally geodesic it follows from the Gauss formula that $g(SX_1 , X_3)g(SX_2 , Y) = g(SX_2 , X_3)g(SX_1 , Y)$ for all $X_1, X_2, X_3 \in \cD$ and $Y \in \cD^\perp$, from which one can deduce that the dimension of $S(\cD)$ is at most one. The above mentioned possibilities occurrence depend on whether $S(\cD)$ is a subspace of $\cD$ or $\cD^\perp$, or is transversal to them. 

Similar to \cite{DajRovToj15} it follows from Theorems~\ref{twisted},~\ref{twisted-partialtube} that if $\alpha_f$ is adapted to $(\cD , \cD^\perp)$ at every point, then $f(M)$ is a partial tube over a curve in $\L^{n+1}$. On the other hand, if $\alpha_f$ vanishes on $\cD$ at every point, then the leaves of $\cD$ turns out to be totally geodesic in $\L^{n+1}$, namely, $f(M)$ is ruled.

So, let the open subset $U \subset M$ of the points at which $\alpha_f$ is neither adapted to $(\cD , \cD^\perp)$ nor vanishes identically on $\cD$, be non-empty. We prove that $U$ contains a $2$-dimensional strip if the relative nullity distribution $\cT$ is not degenerate at every point in $U$. In this case the surface factor of the strip is Riemannian when $M$ is Riemannian and it can be Riemannian or Lorentzian when $M$ is Lorentzian. Another case that can happen when $M$ and the distribution $\cD$ are both Lorentzian is that $\cT$ is degenerate at every point in $U$. In this case, we show that $U$ contains a $3$-dimensional Lorentzian strip. 

Possible forms of the matrix of the shape operator of $M$ are given at the end of the section. One can see that if at a given point the matrix of $S$ is of the form $(ii)$ or $(iv)$ the relative nullity can not be $n-1$ and when it has the form $(i)$ or $(iii)$ then $M$ is flat at the point. So, since $M$ is nowhere flat, the relative nullity space is a subspace of $\cD$ of dimension $n-2$ at any point of $U$ and consequently $U$ is a subset of $W$, where as in the previous section $W \subset M$ denotes the set of the points with minimum relative nullity.  It then follows from a proof similar to Lemma~9 in \cite{DajRovToj15}, that the leaves of $\cD$ and $\cT$ in $U$ are flat and every leaf of $\cT$ intersecting $U$ lies entirely inside $U$.

Denote by $G$ the set of points in $W$ at which the relative nullity distribution is non-degenerate. Then we have:
\begin{proposition}\label{G-open}
$G$ is an open subset of $W$.
\end{proposition}

\begin{proof}
Note that $\im S$, the image of the shape operator $S$ of $M$, coincides with the orthogonal complement of $\cT$ in $TM$. So, $\cT$ is non-degenerate if and only if $\im S$ intersects $\cT$ trivially, or equivalently when $\ker S^2 = \ker S$ which implies that $G$ consists of the points in $W$ with $rank S^2 > 1$.
\end{proof}

The distribution $Rad(\cT) = \cT \cap \cT^\perp$ is a totally null distribution on $W$, namely, $g(X,Y) = 0$ for all $X, Y \in Rad(\cT)$. A screen distribution $\cS$ on $W$ is a complementary distribution of $Rad(\cT)$ in $\cT$, that is, $\cT = \cS \oplus Rad(\cT)$. As we will see cases $f(M)$ contains a strip when there exists a totally geodesic screen distribution on an open subset of $W$. If $G$ is non-empty then $Rad(\cT)$ is trivial on $G$ and according to Theorem~\ref{W-properties}, $\cS = \cT$ is totally geodesic. In Proposition~\ref{t.g screen distribution} we make a totally geodesic screen distribution on an open subset of $W$ when $G$ is empty. 

In the proof of Proposition~\ref{t.g screen distribution} we use a generalized version of the so-called \textit{conullity operator} which we define as follows: let $N$ be a hypersurface in a semi-Riemannian manifold $\bar{M}(c)$ of constant sectional curvature $c$ and let $\cS$ be a totally geodesic sub-distribution of the relative nullity distribution of $N$ with a complement $\cC$ on a neighborhood of a point $p$ in $N$. Denote by $P_\cS$ and $P_\cC$ projections onto $\cS$ and $\cC$ in the decomposition $TN = \cS \oplus \cC$, respectively. Given a vector field $Z$ tangent to $\cS$, we define the linear operator $\c_Z(p) : T_pN \to \cC_p$ by 
\[
\c_Z(p)(X) = - P_\cC (\tilde{\nabla}_XZ) \ ,
\]
where $\tilde{\nabla}$ denotes the induced connection on $N$. The operator $\c$ is a tensor and for $\cS = \cT$ it is known as the conullity operator (see \cite{AbeMag86} and \cite{Fer70b}, for example).

Suppose that $\gamma(t)$, $- \infty < t < + \infty$, is a geodesic in a leaf of $\cS$ passing through $p$ and $Z$ is an extension of the tangent vector field $\gamma'$ of $\gamma$. Similar to the special case $\cS = \cT$ in \cite{AbeMag86}, one can see that the operator $\c_Z$ satisfies a Ricatti type differential equation 
\begin{equation}\label{Ricatti Eq.}
\frac{dy}{dt} = y^2 + c g(Z({\gamma(t)}) , Z({\gamma(t)}))
\end{equation}
along $\gamma$. If the term $c g(Z({\gamma(t)}) , Z({\gamma(t)}))$ vanishes for all $t$, then $\c_Z = 0$ is the only globally defined solution of the equation on $\gamma$.

\begin{proposition}\label{t.g screen distribution}
Under the hypothesis of Theorem~\ref{MainTheorem}, let $G$ be empty. Then there exists a totally geodesic screen distribution on $U$.
\end{proposition}
\begin{proof}
By assumption it turns out that $Rad(\cT)$ is a one dimensional totally null distribution on $M$ which at every point determines the unique one dimensional degenerate subspace of $\cT$. 

The relative nullity distribution $\cT$ is smooth and involutive with complete totally geodesic leaves. So, for a given leaf $L$ of $\cD$, the leaves of $\cT$ passing through the points of $L$ form a codimension one complete degenerate foliation on $L$. Since $L$ is totally geodesic in $M$ and $\cT_p \subset \cD_p$ for all $p \in U$, it follows that the relative nullity distributions of the immersions $f : M \to \L^{n+1}$ and of $f|_L : L \to \L^{n+1}$ coincide on $L$.

Let $p$ be a point in $L$ and $K_p$ denotes the leaf of $\cT$ through $p$. Let $\gamma(t)$ be a regular curve in $L$ defined on $I_\epsilon = (-\epsilon, +\epsilon)$ for some $\epsilon > 0$, which starts at $p$ and be transversal to the leaves of $\cT$. Fix a null vector field $\xi \in Rad(\cT)$ and take a null vector $\hat{\xi}(p) \in \cD_p$ transversal to $K_p$ such that $g(\xi(p) , \hat{\xi}(p)) = 1$. Parallel translate the vector $\hat{\xi}(p)$ along $\gamma(t)$ and then extend it to a vector field $\hat{\xi}$ in a neighborhood of $p$ (in $L$) by parallel translating each $\hat{\xi}({\gamma(t)})$ along the leaf $K_{\gamma(t)}$ of $\cT$ through the point $\gamma(t)$. Note that since $L$ and the leaves of $\cT$ are flat, all parallel displacements are locally path independent. By replacing $\gamma(t)$ with an integral curve of $\hat{\xi}$ we may assume $\gamma'(t)=\hat{\xi}({\gamma(t)})$ for every $t \in I_\epsilon$.

Let $Z$ be a vector field tangent to $\cT$ and $\delta$ be the geodesic in $K_p$ starting at $p$ with $\delta'(0) = Z_p$. Then one has  
\[
\nabla'_XZ = P_\cT(\nabla'_XZ) - \c_ZX 
\]
for all $X \in TL$, where $\nabla'$ is the induced Levi-Civita connection on $L$ and $\c_Z$ is the conullity operator with respect to the decomposition $TL = \cT \oplus \spann\{\hat{\xi}\}$. Since $\c_Z$ satisfies equation~\eqref{Ricatti Eq.} along $\delta$ and knowing that $L$ is flat and $K_p$ is geodesically complete, it follows that $\c_Z=0$. So the parallel displacement of $Z$ is always tangent to the leaves of $\cT$ and for every $X \in TL$ and $Z \in \cT$ we have
\begin{align*}
0 & = X(g(\xi , Z))\\
& = g(\nabla'_X \xi , Z) + g(\xi , \nabla'_XZ) \\
& = g(\nabla'_X \xi , Z)
\end{align*}
which implies that $Rad(\cT) = \spann\{\xi\}$ is a parallel subbundle of $TL$. This allows us to assume that $\xi$ is a parallel vector field along $L$. Let $\cS$ be the orthogonal complement of the Lorentz plane $\spann\{\xi, \hat{\xi}\}$ in $TL$. Then $\cS$ is an $(n-3)$-dimensional space-like distribution in a neighborhood of $p$ in $L$ and we have $\cT = \cS \oplus Rad(\cT)$ since if $a \xi(q) + b \hat{\xi}(q) + v \in \cT_q$ with $a, b \in \R$ and $v \in \cS_q$ for any point $q$ in the neighborhood of $p$, then $g(\xi(q) , a \xi(q) + b \hat{\xi}(q) + v) = 0$ yields $b=0$. 

We saw that for any $X \in TL$ and $Z \in \cS$, $\nabla'_XZ = P_\cT(\nabla'_XZ) \in \cT$. On the other hand, since $\hat{\xi}$ is parallel along the leaves of $\cT$, if $X \in \cS$ then one gets $g(\nabla'_XZ , \hat{\xi}) = - g(Z , \nabla'_X\hat{\xi}) = 0$. So, it follows that $\nabla'_XZ \in \cS$ and thus $\cS$ is a totally geodesic screen distribution on $L$ in a neighborhood of $p$.

Let $Y \in \cD^\perp$ be a unit vector field in a neighborhood of the point $p$ in $U$ and $\theta(s)$ be the integral curve of $Y$ through $p$. Suppose that $\xi$ and $\hat{\xi}$ are smooth null vector fields along $\theta$ such that $g(\xi , \hat{\xi}) = 1$.  We use the above method to extend $\xi(s)$ and $\hat{\xi}(s)$ along the leaf $L_{\theta(s)}$ for each $s$. Then the distribution $\cS \subset TU$ defined locally as the orthogonal complement of $\spann\{Y, \xi, \hat{\xi} \}$ is a totally geodesic screen distribution on $U$. 
\end{proof}

We write $G$ as the disjoint union $G = G_R \cup G_L$ where $G_R$, respectively $G_L$, denotes the set of points in $G$ with Riemannian, respectively Lorentzian, relative nullity and define the open subset $V \subset M$ as follows: if $U \cap G$ is non-empty then $V$ is given by $V = U \cap G_R$ or $V = U \cap G_L$ depending on whether $U \cap G_R$ or $U \cap G_L$ is non-empty, respectively, and $V = U$ if the $U$ and $G$ do not intersect. Let $\cS$ be a totally geodesic screen distribution on $V$ and put $\cC = \cS^\perp$. Note that the leaves of $\cS$ are geodesically complete. In fact, if $\gamma$ is a geodesic in a leaf of $\cT$ such that $\cT_{\gamma(t)}$ is non-degenerate for some $t$, then since $\cT$ is parallel along $\gamma$ it also remains non-degenerate, with the same causality, at every point of the geodesic. Then it follows from completeness of $W$ that the leaves of $\cT$ in $G$ and consequently in $V$ (when $U \cap G$ is non-empty) are complete, as well. On the other hand, if $U \cap G = \emptyset$, then as a totally geodesic sub-distribution of $\cT$, $\cS$ has complete leaves in $U$.  

For vector fields $X \in TM$ and $Z \in \cS$ on $V$ we have $\nabla_XZ = P_\cC(\nabla_XZ) + \c_ZX$, where $\c_Z$ is the conullity operator with respect to $TM = \cS \oplus \cC$. Since the leaves of $\cS$ are complete and $\L^{n+1}$ is flat we get $\c_Z = 0$. So, $\nabla_XZ \in \cS$ and therefore $\cS$ is parallel in $TM$. Then it follows from the equation $\bar{\nabla}_Xf_*Z = f_*(\nabla_XZ) + \alpha_f(X , Z) = f_*(\nabla_XZ)$ that $f_*\cS$ is parallel in $f^*T\L^{n+1}$. Moreover, for $Z \in \cS$ and $X,Y \in \cC$ one gets $Yg(X , Z) = g(\nabla_YX , Z) + g(X , \nabla_YZ) = 0$ which since $\nabla_YZ \in \cS$ yields $g(\nabla_YX , Z) = 0$. This means that $\cC$ is totally geodesic. Now it follows that $V$ is isometric to $\R^{n-k} \times L^k$ and the immersion $f$ decomposes into $f = id \times g$ where $L^k$ in each connected component of $V$ is a maximal integral leaf of $\cC$ of dimension $k = rank \cC$ and $g: L^k \to \R^{k+1}_s$, $s = 0, 1$, is an isometric immersion. Indeed, when $U \cap G \neq \emptyset$, then at least one of $U \cap G_R$ or $U \cap G_L$ is also non-empty. When $V = U \cap G_R$ is non-empty we have $id  = id_{\L^{n-2}}$ and $g : L^2 \to \R^3$ is a Riemannian surface in the above decomposition of $f$ on $V$, while if $V = U \cap G_L$ is non-empty, then $id  = id_{\R^{n-2}}$ and $g : L^2 \to \L^3$ is a Lorentzian surface. Finally, if $V = U$ then $id = id_{\R^{n-3}}$ and $g : L^3 \to \L^4$ is a Lorentzian hypersurface in $\L^4$. Here is the end of the proof of Theorem~\ref{MainTheorem}.

The given proof of Theorem~\ref{MainTheorem} is much involved and it can be used to get other results in relevant contents. In the following we consider embedded, Einstein or conformally flat hypersurfaces in the Minkowski space. 
\begin{theorem} \label{embedded}
Let $M^n \subset \L^{n+1}$ be an embedded semi-Riemannian hypersurface that carries a semi-Riemannian totally geodesic foliation of codimension one with complete leaves. Then $M$ is a partial tube over a curve.
\end{theorem}
\begin{proof}
Since $M$ is an embedded hypersurface it is locally a level set of a smooth function on $\L^{n+1}$. For a given point $p \in M$, let $\bar{U}$ be a neighborhood of $p$ in $\L^{n+1}$ and $h : \bar{U} \to \R$ a smooth function with non-vanishing gradient such that $U = M \cap \bar{U} = h^{-1}(c)$ for some $c \in \R$. By shrinking $\bar{U}$ if necessary, we may assume that there exists a smooth $n$-frame $\{\bar{X}_1, \ldots, \bar{X}_n \}$ on $\bar{U}$ which generates $TM = \ker dh$ at every point of $U$. Let $Y$ be a unit vector field on $U$ tangent to $M$ and normal to the foliation. One can write $Y = \Sigma_{i = 1}^n h_iX_i$ for smooth functions $h_i \in C^\infty(U)$ where $X_i = \bar{X}_i|_U$. We extend $Y$ on a neighborhood of $p$ in $\L^{n+1}$ by extending its coefficients $h_i$ and in this way we get $dh(Y) = 0$ and thus $g(Y , \eta) = 0$, where $\eta = |g(\bar{\nabla}h , \bar{\nabla}h)|^{-\frac{1}{2}} \bar{\nabla}h$. For every vector field $X$ tangent to the foliation we have
\begin{align*}
0 = X g(Y , \eta) & =  g(\bar{\nabla}_X Y , \eta) + g(Y , \bar{\nabla}_X\eta) \\
& =  g(\nabla_XY , \eta) + g(\alpha(X , Y) , \eta) + g(\alpha(X , Y) , \eta) \\
& =  2g(\alpha(X , Y) , \eta) .
\end{align*}
where $\alpha$ is the second fundamental form of $M$. So, it follows that $\alpha$ is adapted to $(\cD , \cD^\perp)$ and as in the proof of Theorem~\ref{MainTheorem}, in this case $M$ is a partial tube over a curve in $\L^{n+1}$.
\end{proof}

Let $T$ be a symmetric endomorphism on an $n$-dimensional vector space endowed with a Lorentzian scalar product. Then according to \cite{Pet69}, the matrix of $T$ has one of the the following four forms:
\begin{align*} 
& (i): 
\begin{pmatrix}
a_1 & & 0 \\  &  \ddots & \\ 0 & & a_n
\end{pmatrix}, & 
& (ii): 
\begin{pmatrix}
\begin{matrix} a_0 & b_0 \\ -b_0 & a_0 \end{matrix} & & 0 \\  &  \begin{matrix} a_1 &  \\  & \ddots \end{matrix} & \\ 0 & & a_{n-2}
\end{pmatrix} , \\  
&(iii): 
\begin{pmatrix}
\begin{matrix} a_0 & 0 \\ 1 & a_0 \end{matrix} & & 0 \\  &  \begin{matrix} a_1 &  \\  & \ddots \end{matrix} & \\ 0 & & a_{n-2}
\end{pmatrix} , & 
&(iv):  
\begin{pmatrix}
\begin{matrix} a_0 & 0 & 0 \\ 0 & a_0 & 1 \\ -1 & 0 & a_0 \end{matrix} & & 0 \\  &  \begin{matrix} a_1 &  \\  & \ddots \end{matrix} & \\ 0 & & a_{n-3}
\end{pmatrix} 
\end{align*}
where in cases $(i)$ and $(ii)$, $T$ is represented in an orthonormal basis and in cases $(iii)$ and $(iv)$ it is written with respect to a pseudo-orthonormal basis. 

Recall that a pseudo-Riemannian manifold $(M,g)$ is called Einstein if its Ricci tensor satisfies $Ric_M = \lambda g$ for some some $\lambda \in \R$ and $M$ is said to be conformally flat, if each point of $M$ belongs to a neighborhood $U \subset M$ such that, for a certain smooth function $f$ on $U$, the submanifold $(U, e^fg)$ is flat. 

Let $f: M^n \to \L^{n+1}$, $n \geq 4$, be an Einstein or conformally flat hypersurface and there exists a semi-Riemannian totally geodesic distribution $\cD$ of codimension one on $M$ with complete leaves. According to \cite{VanVer87} if $M$ is Einstein then at any point the matrix of its shape operator $S$ has one of the following forms: it is a scalar multiple of the identity matrix, or diagonal with at most one non-zero entry, or of the form $(iii)$ with $a_0 = ... = a_{n-2} = 0$. In all cases the index of the relative nullity belongs to $\{0, 1, n-1, n\}$ and thus we may assume that $\alpha_f$ is adapted to $(\cD , \cD^\perp)$, that is $S(\cD) \subset \cD$, or $\alpha_f$ vanishes on $\cD$, namely $S(\cD) \subset \cD^\perp$, since otherwise as we proved in Section~\ref{ProofMainTheorem} the index of the relative nullity would be $n-2$. Assume that at a given point $p$ in $M$, $S$ is a scalar multiple of the identity. Then clearly it preserves the distribution $\cD$. Also, if $S$ is diagonal with the only non-zero entry $(S)_{ii} = a$ for some $1 \leq i \leq n$, we get $S(\cD) \subset \cD$ since if $S(\cD) \subset \cD^\perp$ then one gets $g(Sx , x) = ax_i^2 = 0$ for any $x = (x_1, \ldots, x_n)^t \in \cT_p$. This implies $x_i = 0$ and therefore $S(\cD) = 0$. Similarly, if the matrix of $S$ is of the form $(iii)$ with $a_0 = ... = a_{n-2} = 0$ it follows that $S$ preserves $\cD$. 

For a conformally flat hypersurface, as shown in \cite{Gue03}, the matrix of its shape operator either takes the form $(i)$ with the $a_1 = \cdots = a_n$ or the form $(iii)$ with $a_0 = \cdots = a_{n-2}$. Analogous to the above argument for Einstein hypersurfaces, one can see that in this case the second fundamental form $\alpha_f$ is adapted to $(\cD , \cD^\perp)$, as well. So, we have proved
\begin{proposition}\label{Ein.Conf}
Let $f: M^n \to \L^{n+1}$, $n \geq 4$, be an Einstein or conformally flat hypersurface that carries a semi-Riemannian totally geodesic foliation of codimension one with complete leaves. Then $M$ is a partial tube over a curve.
\end{proposition}

\providecommand{\href}[2]{#2}
\bibliographystyle{amsplain}

\end{document}